\documentclass[a4paper]{article}

\usepackage[utf8]{inputenc}
\usepackage[T2A]{fontenc}
\usepackage[english]{babel}

\usepackage[a4paper,top=3.5cm,bottom=3.5cm,left=3.2cm,right=3.2cm,marginparwidth=1.75cm]{geometry}

\usepackage{amsmath, amsthm, amssymb}
\usepackage{authblk}
\usepackage{url}
\usepackage{tikz}
\usepackage{graphicx}
\usepackage[colorinlistoftodos]{todonotes}
\usepackage[colorlinks=true, allcolors=blue]{hyperref}

\newtheorem{thm}{Theorem}

\newtheorem{cor}[thm]{Corollary}
\newtheorem{prop}[thm]{Proposition}

\theoremstyle{definition}
\newtheorem{axiom}{Axiom}
\newtheorem{defn}{Definition}
\newtheorem{exampl}{Example}
\newtheorem*{rem}{Remark}
\theoremstyle{remark}

\AtEndDocument{%
	\par
	\medskip
	\begin{tabular}{@{}l@{}}%
		\textsc{Skoltech and Higher School of Economics}\\
		\textit{E-mail address}: \texttt{nklemyatin@yandex.ru}
\end{tabular}}

\title{Universal spaces of parameters for complex Grassmann manifolds $G_{q+1,2}$.}
\author[1]{Nikita Klemyatin}
\affil{Skoltech and Higher School of Economics}
\date{}

\begin{document}

\fontsize{13}{16pt}\selectfont
\renewcommand{\proofname}{$\bigtriangleup$}
\newcommand{\Natural}{\mathbb N}
\newcommand{\Integer}{\mathbb{Z}}
\newcommand{\Real}{\mathbb R}
\newcommand{\Complex}{\mathbb C}
\newcommand{\cp}{\Complex P}
\newcommand{\pf}{(\cp^1)^5}
\newcommand{\ipf}{\cp^1)^5 \setminus \Delta}
\newcommand{\pgl}{PGL(2, \Complex)}
\newcommand{\moduli}{\overline{M}_{0,n} }	
\newcommand{\target}{(\cp^1)^N}	
\newcommand{\grass}{G_{q+1,k}}
\newcommand{\gtwo}{G_{q+1,2}}
\newcommand{\gln}{GL(n; \Complex)}
\newcommand*{\Hom}{\mathrm{Hom}\kern -.5pt }
\newcommand{\matkn}{M_{q+1,k}(\Complex)}
\newcommand{\matknm}{M^{max}_{q+1,k}(\Complex)}
\newcommand{\hprsmplx}{\Delta_{q+1,2}}

\maketitle

\begin{abstract}
	\normalsize
	In \cite{BT1} and \cite{BT2} Buchstaber and Terzic introduced a notion of universal space of parameters $\mathcal{F}$ for a manifold $M^{2n}$, which has an effective action of compact torus $T^k$ , $k \leq n$ with some additional properties. with special properties. This space is needed to construction of factor $M^{2n}/T^k$. Buchstaber and Terzic constructed the universal space of parameters for $G_{5,2}$ in \cite{BT1}. In this work we construct universal space of parameters for complex Grassmann manifold $\gtwo$. Our construction is based on the construction of moduli space of stable curves of genus zero with $q+1$ marked points due to Salamon, McDuff and Hofer.  
\end{abstract}

\tableofcontents

\section{Introduction.}

The complex Grassmann manifolds are fundamental objects in various branches of mathematics, such as differential and algebraic geometry, algebraic topology, representation theory. These manifolds are important either as examples due to their simple -- and at the same time -- rich geometry or as classifying spaces in  algebraic topology. Also they are important examples of spherical varieties in representation theory.  

Since the complex Grassmann manifolds are homogeneous spaces, they posseses an action of torus. Studying this action is important in symplectic geometry and toric topology. The canonical action of torus on $G_{q+1,k}$ is one of the simplest example of action of positive complexity. In this case, studying the orbit space and moment map is harder, because additional parameters appear. For example, in the case of $G_{q+1,2}$ the complexity is equal to $\dim_\Complex G_{q+1,2} - q = 2q- q -2 = q-2$.

However, until quite recently, the question about the factor $G_{q+1,k}/T^q$ could not be solved due to abscenсe of methods for description of topology structure of it. In order to describe $G_{q+1,2}/T^q$, Buchstaber and Terzic in \cite{BT1} and \cite{BT2} introduced notions of $(2n,k)$-manifold and virtual spaces of parameters. They also proved that  $G_{4,2}$ and $G_{5,2}$ are $(8;3)$ and $(12;4)$ manifolds in the sense of this definition. Furthermore, they are proved that all $G_{q+1,2}$ for $q \geq 5$  satisfy almost all axioms of $(2n,k)$-manifolds except the last one. The last axiom says that there is a compact manifold, which consists of all virtual spaces of parameters and this compact manifold is a compactification of space of parameters for the main stratum.

In this work we prove that for each $G_{q+1,2}$ the manifold, so-called “Chow factor” $G_{q+1,2} // (\Complex^*)^q$ is a universal space of parameters. As a corollary, we obtain, that for each $q$ the Grassmann manifold $G_{q+1,2}$ is $(4(q-1); q)$ manifold in the sense of Buchstaber and Terzic.

I am very grateful to Victor Buchstaber for turning my mind into this problem and fruitful discussions.
I am also very grateful to Anton Ayzenberg, Vladislav Cherepanov for valuable discussions. I also very grateful to Alexei Rukhovich and Sergey Khakhalov for helping with this preprint.

\section{Torus action on $\gtwo$, equivariant stratification and spaces of parameters of strata.}

\subsection{Action of $T^{q+1}$ and $(\Complex^*)^{q+1}$ on $\gtwo$.}

\begin{defn}
	A \emph{Grassmann manifold} (or Grassmannian) $\gtwo$ is a space which parametrizes all 2-dimensional linear subspaces of the (q+1)-dimensional vector space $\Complex^{q+1}$.
\end{defn}

Any element $L \in \gtwo$ may be represented by matrix $A_L$, which columns forms a basis in $L$:

	\[
A_L =
\begin{bmatrix}
a_1  & b_{12}   \\
a_2  & b_2 	  \\
\vdots & \vdots \\
a_n & b_n \\
\end{bmatrix}.		
\]

This matrix is not unique, but any other such matrix $B_L$ may be obtained from $A_L$ by right action of $GL(2, \Complex)$.

Let $P_{ij}= a_i b_j - a_j b_i$.

\begin{defn}
	$P_{ij}$ is called a \emph{Plucker coordinate}. 
\end{defn}

\begin{prop}
	\begin{enumerate}
	\item All Plucker coordinates are defined up to common nonzero factor. 
	
	\item All Plucker coordinates define an embedding of $\gtwo$ into $cp^{\binom{q+1}{2} - 1}$;
	
	\item Plucker coordinates are satisfy the \emph{Plucker relations}:
	
	$$P_{ij}P_{kl} - P_{ik}P_{jl} + P_{jk}P_{il}=0. $$ 
	\end{enumerate}
\end{prop}

For the proof see \cite{GH} or \cite{Fltn}.

This is a classical result, that Grassman manifold $\gtwo$ of two-planes in $\Complex^{q+1}$ may be viewed as homogeneous space of groups $U(n)$ and $GL(q+1, \Complex)$. Really, both groups acts transitivilly on the set of 2-planes in $\Complex^n$. In the case of $U(n)$ the stabilizer of $span_\Complex\{e_1, e_2 \}$ is equal to $U(2) \times U(n-2)$ (here $e_1, e_2$ are vectors from standard basis of $\Complex^n$). So, we have $$\gtwo = U(q+1) /U(2) \times U(q-1).$$
From this point of view one can see that the torus $T^{q+1} \subset U(q)$ acts on $\gtwo$. This action is an example of Hamiltonian action from symplectic geometry. The easiest way to see it is follows: as we mentioned before, there is the Plucker embedding $\Phi:\gtwo \rightarrow \cp^{\binom{q+1}{2} - 1}$ is $T^{q+1}$-equivariant. Hence, the pull-back $\Phi^* \omega_{FS}$ of Fubini-Studi form is Kahler form on $\gtwo$. It is not hard to show that the action of $T^{q+1}$ is actually Hamiltonian. Moreover, one can show, that this metric is also $U(q+1)$-invariant.

There is a $T^{q+1}$-equivariant map $\mu: \gtwo \rightarrow \Delta_{q+1,2} = \{x \in \Real^n ~|~ 0 \leq x_i \leq 1, ~ x_1 + \dots + x_n = 2  \}$.
We are going to write the formula for $\mu$. Denote by $e_i$ standard basis of $\Real^{q+1}$. Let $e_{ij} := e_i + e_j$ and $P(L):=\sum_{i < j} {\lvert P_{ij} \rvert}^2$. In these definitions the formula for $\mu$ is very simple:

$$\mu(L) = \frac{1}{P(L)} \sum {\lvert P_{ij} \rvert}^2 e_{ij}. $$

It's not hard to see that the map is $T^n$-equivariant and maps $\gtwo$ onto $\Delta_{q+1,2}$. One could show that $\mu$ is moment map for Kahler metric $\Phi^* \omega_{FS}$, which was mentioned above.

There is another way to represent $\gtwo$ as a homogeneous space. As we noticed before, $GL(n,\Complex)$ acts transitively on 2-planes in $\Complex^n$. The stabilizer (it is parabolic subgroup of $GL(n,\Complex)$) consists of matrices 

\[
A=
\left[ {\begin{array}{cc}
	A_0 & A_1 \\
	0   & A_2 \\
	\end{array} }\right].	
\] 

Here $A_0 \in GL(2, \Complex)$, $A_2 \in GL(n-2, \Complex)$ and $A_1$ is an arbitrary complex matrix with two rows and $(n-2)$ columns. So, the complex torus $H:=(\Complex^*)^n \subset GL(n, \Complex)$ acts on $\gtwo$.

Here we describe a coordinate way. We can represent any $L \in \gtwo$ by $2 \times n$ matrix: 

	\[
A_L =
\begin{bmatrix}
a_1  & b_1   \\
a_2  & b_2 	  \\
\vdots & \vdots \\
a_{q+1} & b_{q+1} \\
\end{bmatrix}.		
\]

Let $t=(t_1,\dots, t_n) \in H$. The action of $H$ describes as follows:

 	\[
 t A_L =
 \begin{bmatrix}
 t_1 a_1  & t_1 a_{12}   \\
 t_2 a_2  & t_2 a_{22} 	  \\
 \vdots & \vdots \\
 t_{q+1} a_{q+1} & t_{q+1} b_{q+1} \\
 \end{bmatrix}.		
 \]

\begin{rem}
	Notice, that the diagonal subgroups of $T^n$ and $H$ both act trivially on $\gtwo$.
\end{rem}

Denote $\tau_{ij} = t_i t_j$. It is easy to see that we have an equality: $P_{ij}(\tau L) = \tau_{ij} P_{ij}(L)$. 
Note, that $\frac{\tau_{ij}}{\tau_{ik}} = \frac{t_j}{t_k}$. Hence, we can reconstruct all $t_j$ by $\tau_{ij}$ up to a common factor. Since the common factor is just an element of diagonal of $H$, we can completely rebuild torus action up to action of diagonal (which acts trivially).
One could notice that an element $t \in H$ fixes $L$ iff $t_i t_j P_{ij}(L) = \lambda P_{ij}(L)$ for some $\lambda \in \Complex^*$. Without loss of generality we can assume that $\lambda=1$.

\begin{rem}
	All these results are true for arbitrary $G_{n,k}$ (with obvious modifications).
\end{rem}

\subsection{Stratification of $\gtwo$ and spaces of parameters of strata.} \label{sect12}

In this section we define the notion of strata on $\gtwo$. We give two definitions of stratification and show their equivalence.

Let $K = \{I \subset 2^{\{1,\dots,n  \}  } | ~|I|=2 \}$. Suppose $Y_I = G_{n,k} \setminus M_I ~\forall I \in K$. Here $M_I$ -- standard coordinate chart on $\gtwo$. Suppose also $\sigma = \{I_1,\dots, I_l  \}$ and all $I_j \in K$.

\begin{defn}
For all $\sigma$ we define $$W_\sigma = (\cap_{I \in \sigma} M_I) \bigcap (\cap_{I \in K \setminus \sigma} Y_I).$$ The set $W_\sigma$ (if it's non-empty) is called a \emph{stratum}. Set $W = \bigcap M_I$ is called the \emph{main stratum}.
\end{defn}

This definition of stratum was given by Buchstaber and Terzic in\cite{BT1} and \cite{BT2} for so-called $(2p;q)$ manifolds.

One can notice that $G_{n,k} = \cup W_\sigma$ and $W_\sigma \bigcap W_\eta = \emptyset$. We also notice, that all strata are invariant under $T^n$ and $H$ action.

Now we give another definotion of stratification. This definition was given by Gelfand and MacPherson \cite{GM}. This definition may be found in the paper of Kapranov \cite{Kap}.

\begin{defn}
	Planes $L_1$ and $L_2$ from $\gtwo$ lines in the same stratum iff $\mu(\overline{H.L_1}) = \mu(\overline{H.L_2}) = P$. Here $P \subset \Delta_{n,k}$ -- convex polytope, whose vertices are among vertices of $\Delta_{n,k}$. More precisely, a stratum $\tilde{W}$ consist of all $L$, such $\mu(\overline{H.L}) = P$. 
\end{defn}

This stratification is also $H$ and $T^n$ invariant.
So, we have a question about connection between these two stratifications.

\begin{prop}
	For $\gtwo$ both stratifications are the same.
\end{prop}

\begin{proof}[Proof]
	In the definition of Buchstaber and Terzic all sets $Y_I$ are defined by equation $P_I = 0$. So, this definition is equivalent to the next one: we says, for which $I$ Plucker coordinates are equal zero.
	
	On the other hand, a polytope $P$ (from another definition) is a convex hull of it's vertices. If $e_I \in \Delta_{n,k}$ is not a vertice of $P$, then one can see from formula for moment map that all elements of stratum should satisfy the equality $P_I = 0$. The converse is also true. Hence, both definitions are equivalent follows: stratum $\tilde{W}$ is defined by indication, for which $I$ Plucker coordinates $P_I$ is equal zero.
\end{proof}

Further, we will use notions of Buchstaber and Terzic.

We also need definition of a spaces of parameters.

\begin{defn}
	For any stratum $W_\sigma$ we define it's \emph{space of parameters} ar $F_\sigma = W_\sigma / H$. We also define an \emph{admissibe} polytope of stratum as a convex polytope $P_\sigma$, such $\mathring{P}_\sigma = \mu (W_\sigma)$.
\end{defn}

This definition was given in \cite{BT1}. In this article Buchstaber and Terzic also showed that for any stratum there is a homeomorphism $h_\sigma: W_\sigma / T^n \rightarrow \mathring{P}_\sigma \times F_\sigma$, defined by maps $\mu: W_\sigma/T^n \rightarrow \mathring{P}_\sigma$ and by canonical projection $p_\sigma:W_\sigma/T^{q+1} \rightarrow W_\sigma /H$.

\subsection{Admissible polytopes for the hypersimplex $\Delta_{q+1,2}$.}

As we mentioned before, the hypersimplex $\hprsmplx$ is the moment polytope for $\gtwo$. It has a very simple description:

$$\hprsmplx =  \{x \in \Real^n ~|~ 0 \leq x_i \leq 1, ~ x_1 + \dots + x_n = 2  \}.$$

We need to understand which subpolytopes in $\hprsmplx$ may be an admissible. 

\begin{prop}
	The polytope from boundary of admissible polytope is admissible.
\end{prop}

This is an obvious and we omit the proof.

\begin{cor}
	Any polytope in $\partial \hprsmplx$ is admissible.
\end{cor}

The polytopes in boundary of $\hprsmplx$ has a very nice description: it's hypersimplices, defined by equations $x_i = 0$ or $x_i = 1$. 

There is a polytopes of codimension one in $\hprsmplx$, defined by equation $l_I(x) = x_{i_1} + \dots + x_{i_k} = 1$ for some $k \geq 2$ and $I=\{i_1, \dots, i_k \}$. As Kapranov proved in \cite{Kap}, all admissible polytopes of codimension 1 are either lies in $\partial\hprsmplx$ or defined by the form $l_I$ for some $I$.

Now suppose that a plane $L \in \gtwo$ does not lie in any $L_i$ for any $i$ and $L$ has non-trivial stabilizer in $(\Complex^*)^{q+1}$.

Here is an obvious proposition.

\begin{prop}
	 A plane $L \in \gtwo$ does not lie in any $L_i$ iff for each $i \in [q+1] := \{1, \dots, q+1\}$ there is a $j \neq i$, such as $P_{ij}(L) \neq 0$.	 
\end{prop} 

Now we can construct an equivalence relation on $[q+1]$, which needs to describe the moment polytope for $H.L$.

\begin{thm}\label{ThmSix}
	Suppose that $L \in \gtwo$ does not lie in any $L_i$. Then there is an an equivalence relation on $[q+1]$ with two equivalence classes $I_1$ and $I_2$, which is defined by $L$. The moment polytope for $H.L$ is a product of two simplices, is defined by the formula:
	
	$$\sum_{i \in I} x_i = 1,$$
	
	and $I$ is either $I_1$ or $I_2$. 
\end{thm} 

\begin{proof}[Proof]
	We will say that $i \sim j$ iff $P_{ij}(L) = 0$ (and, formally, $P_{ii} = 0$). This relation obviously have the reflexive property and the symmetric property. We only need to check the transitive property. 
	
	Suppose that $i \sim j$ and $j \sim k$. Since $L$ does not lie in any coordinate hyperplane, we can choose an index $l$, such as $P_{jl}(L) \neq 0.$ Now, by Plucker identity
	
	$$P_{ij}P_{kl} - P_{ik}P_{jl} + P_{jk}P_{il}=0, $$ 
	
	and by the fact that $P_{ij}(L) = P_{jk}(L) = 0$, we can conclude that $P_{ik}(L)P_{jl}(L) = 0$. Hence $P_{ik}(L) = 0$ and $i \sim k$. So, this is an equivalence relation.
	
	Let $I$ be any equivalence class, for this equivalence relation. We can show, that
	a polytope, which is defined by the equation
	
	 $$q_I(x):=\sum_{i \in I} x_i = 1$$ 
	 
	 consists of the moment polytope $P_L$ corresonding $H.L$. Really, any admissible polytope spanned on vertices of hypersimplex $\hprsmplx$. The verticle $e_{ij}$ lies in $P$ iff $q_I(e_{ij})=1$. The last equality holds iff $i \nsim j$.  
	
	We also have that all $k \in J=[q+1] \setminus I$ should be equivalent. Really, from equations $x_1 + \dots + x_n = 2 $ and $\sum_{i \in I} x_i = 1$ we have, that $q_J(x)=\sum_{k \in J} x_i = 1$ for any $x \in P$. Hence  $q_J(x)=1$ also defines the same hyperplane section of $\hprsmplx$ as $q_I(x)=1$. If $k,l \in J$ not equivalent, then $P_{kl}(L) \neq 0$ and $e_{kl}$ is a verticle of $P_L$. But in this case $q_J(e_{kl}) =2$. Contradiction.
	
	It's easy to see that $P_L$ is not only lies in intersection $\hprsmplx$ with hyperplane $q_I(x) = 1$, but it's coinside with this intersection. It's obvious that this intersection is nothing, but product of two simplices $\Delta^{\sharp I - 1} \times \Delta^{\sharp J - 1}$.
	
\end{proof}

	Note that this theorem is actually true not only for some $L$, but for whole stratum, which consists $L$.

\begin{cor}
	For any stratum $W_\sigma \subset \gtwo$ the moment polytope $P_\sigma$ is either a (hyper)simplex or a product of some simplices.	
\end{cor}

\begin{proof}[Proof]
	Assume that $P_\sigma $ is not hypersimplex. Without loss of generality we can assume that $W_\sigma$ consists of planes, which does not lie in any coordinate subspace. Now we can apply the previous theorem.
\end{proof}

\begin{rem}
	By theorem of Kapranov \cite{Kap}, any polytope, which is defined by some linear form $q_I(x)$ is actually a boundary of some admissible polytope. He also showed, that any admissible polytope of codimension one, which lie in $\hprsmplx$ is either a polytope from $\partial \hprsmplx$ or a polytope, defined by the equation $q_I(x) = 1$ for some $I$.
\end{rem}

\subsection{Axioms and examples of $(2n,k)$-manifolds.}
In this section we enlist $6$ axioms of $(2p,q)$-manifolds. These axioms were given in \cite{BT2}. 

Let $M$ is a compact oriented simple-sonnectd manifold, $\dim_\Real (M) = 2n$. Suppose we have an action $\theta$ of torus $T^q$. Suppose also we have an equivariant map $\mu: M \rightarrow \Real^q$, with trivial $T^q$ action on $\Real^q$ and the image of $\mu$ is a convex polytope $P^q$. Triple $(M, \theta, \mu )$ called $(2n,k)$-manifold, if it satisfies next $6$ axioms. 

\begin{axiom}
	There is a smooth atlas of charts $\{M_i,\phi_i \}$ on $M$, $\phi: M_i \rightarrow \Real^{2p} = \Complex^p$, such all charts are an invariant under $T^q$ action and consists exactly one fixed point $x_i$. Moreover $\phi(x_i) = 0$ and any chart $M_i$ is dense in $M$.
\end{axiom}

\begin{axiom}
	The map $\mu$ gives a bijection between fixed points of $T^k$-action and vertices of $P^k$.
\end{axiom}

Next, we will use definition of stratum from the previous section.

Let $\Sigma$ be set of all admissible sets (set called \emph{admissible} if corresponding stratum is non-empty). define a map $s: \Sigma \rightarrow S(P)$, which maps each admissible set from $\sigma$ into polytope $P_\sigma = conv\langle \mu(x_{i_1}) ; \dots; \mu(x_{i_l})\rangle$. We will call such polytopes an admissible polytopes (there is no contradictions with previous section).

\begin{defn}
	Let $S(T^k)$ -- set of all connected subgroups in $T^k$. The map $\chi: M \rightarrow S(T^k)$,$~\chi: x \mapsto Stab(x)$ called \emph{characteristic function}.
\end{defn}

Now we can formulate next axiom.

\begin{axiom}
	Characteristic function is constant on each stratum $W_\sigma$.
\end{axiom}

Denote $T^\sigma := T^k/\chi(W_\sigma)$.

Notice, that $\mu: M \rightarrow P^k$ induce a map $\hat{\mu}: M/T^k \rightarrow P^k$.

\begin{axiom}
	Almost moment map $\mu$ should satisfy next properties:
	\\~
	\\
	1. $\mu(W_\sigma) \subset \mathring{P}_\sigma;$
	\\~
	\\
	2. $\hat{\mu}:W_\sigma/T^\sigma \rightarrow \mathring{P}_\sigma $ if locally trivial fibration;
	\\~
	\\
	3. $\dim P_\sigma = \dim T^\sigma.$
	\\~
\end{axiom}

\begin{defn}
	The fiber $F_\sigma$ of locally trivial fibration $\hat{\mu}: W_\sigma/T^\sigma \rightarrow \mathring{P}_\sigma $ is called \emph{space of parameters} of $W_\sigma$. 
\end{defn}

\begin{rem}
	This definition is the same for the case of $\gtwo$.
\end{rem}

\begin{axiom}
	For all stratum $W_\sigma$ the boundary of  leaf$\partial W_\sigma [\xi_\sigma, c_\sigma]$ is union of leafs $W_{\overline{\sigma}}$, such as $P_{\overline{\sigma}}$ is facet $P_\sigma$.
\end{axiom}

\begin{axiom} \label{A6}
	There exist a topological space $\mathcal{F}$ and subspaces $\tilde{F}_\sigma \subset \mathcal{F}$, such as:
	\begin{enumerate}
		\item For the main stratum $W$ there is an equality $\tilde{F} = F$. Here $F$ -- space of parameters of the main stratum;
	
		\item $\mathcal{F}$ is a compactification of $F$;

		\item $\mathcal{F} = \cup_\sigma \tilde{F}_\sigma$;
		
		\item If $P_{\sigma_1}$ lies in $\partial P_\sigma$, then $\tilde{F}_\sigma  \subset \tilde{F}_{\sigma_1}$;

		\item For all $\sigma$ there exist $p_\sigma: \tilde{F}_\sigma \rightarrow F_\sigma$;

		\item The map $\mathcal{H}: \cup_\sigma \mathring{P}_\sigma \times \tilde{F}_\sigma \rightarrow \Delta_{q+1,2} \times \mathcal{F}$, defined as $\mathcal{H}(x,c) = (x; p_\sigma(c))$ is continious map.
	\end{enumerate}
	
\end{axiom}

Examples of $(2n,k)$-manifolds is spheres (type of $(2n,1)$) and quasitoric manifolds (type of $(2n,n)$). Also, Grassman Manifolds $G_{4,2}$
and $G_{5,2}$ are also such manifolds of types $(8;3)$ and $(12;4)$ respectfully. It was showed by Buchstaber and Terzic in \cite{BT3} and \cite{BT1}. For other $G_{q+1,2}$ with $q \geq 5$ it was unknown, because there is no proof of veracity of axiom 6 in this case.

\subsection{Definition of Chow factor.}
Here we briefly describe the notion of Chow factor of a projective variety $X$ which possesses an algebraic group $G$. One can find further details and references in \cite{Kap}.

Suppose we have a projective variety $X$ over $\Complex$ and an algebraic group $G$,which acts on $X$. Suppose also, that there is an open subset $U \subset X$, such $G$ acts freely on $U$. Then for any $x \in U$ the closure of it's $G$-orbit $\overline{G.x}$ (suppose that dimension of such cycle is equal $r$) is an algebraic cycle. Moreover, all cycles, obtained this way has the same degree (as algebraic subvariety in $X$) and represent the same homology class $\delta$ in $H_{2r}(X, \Complex)$. Obviously, all these cycles are parametrized by $U/G$. Hence we have an embedding of $U/G$ into $\mathcal{C}_r(X;\delta)$ -- the set of all cycles in $X$, which has fixed dimension $r$ and represents homological class $\delta$. $\mathcal{C}_r(X;\delta)$ has a structure of complex (even projective) manifold (look \cite{Kap}, chapter 0.1 and Barlet paper \cite{Barlet}), что $\mathcal{C}_r(X;\delta)$.

So, we have an important definition.

\begin{defn}
	\emph{Chow factor} $X // G$ of variety $X$ is closure of $U/G$ in $\mathcal{C}_r(X;\delta)$.
\end{defn}

\begin{rem}
	Obviously, this is not only way to construct compactification of $U/G$. There is an embedding of $U/G$ into Hilbert scheme or famous GIT-factor of $X$. Kapranov in \cite{Kap} showes that in the case of reductive $G$ there is a birational map from Chow factor to the GIT-factor. Kapranov also shows, if $X = \grass$ and $G = H = (\Complex^*)^n$, then Chow factor is isomorphic to compactification of $U/G$ in Hilbert scheme.
\end{rem}

In the case of $X = \gtwo$ and $G = H$ there is an important theorem, proved by Kapranov (\cite{Kap}, theorem 4.1.8):

\begin{thm}
	The Chow factor $\gtwo // H$ is isomorphic to the moduli space $\moduli$ of stable curves of genus zero with n marked points.
\end{thm}

\section{Cross-ratios on Grassmanians $\gtwo$.}

\subsection{Definition of cross-ratios and its properties.}

In this section we define and describe some objects, named  cross-ratios. Cross-ratios are a crucial element of our construction, so it is necessary to study they before proving the main theorem.

\begin{defn}
	A cross-ratio is a meromorphic function on $\gtwo$ defined by the next formula:
	$$w_{i,j,k,l} := \frac{P_{ik}P_{jl}}{P_{il}P_{jk}}. $$
\end{defn}

It is not so hard to check that total number of cross-ratios on $\gtwo$ is equal $\binom{q+1}{4}$. 

There is some useful formulas for cross-ratios.

\begin{prop} $\\$

	\begin{enumerate} 
		\item $ w_{i,j,k,l} = w^{-1}_{j,i,k,l} ;$
		
		\item $1 - w_{i,j,k,l} = -w_{i,k,j,l};$
		
		\item $w_{m,j,k,l} = \frac{w_{i,j,m,k} - 1}{w_{i,j,m,k} - \Phi_{i,j,l,k}};$
		
		\item $\forall i,j,k,l,m : ~w_{i,j,k,l}w_{i,j,k,m}^{-1}w_{i,j,l,m} = 1.$
	\end{enumerate}
\end{prop}

\begin{proof}[Proof]
	All these formulas will be proved by direct computations and using Plucker identities.
	\begin{enumerate}

		\item This statement directly follows from the definition:	$$w_{i,j,k,l} = \frac{P_{ik} P_{jl} }{P_{il}P_{jk}}$$ and
		$$w_{j,i,k,l} = \frac{P_{il} P_{jk} }{P_{ik}P_{jl}}.$$

		\item For proving this we need a Plucker identities: $P_{jm }P_{il} - P_{im} P_{jl} = P_{lm} P_{ij}$. Now we have:
		
		$$ 1 - w_{i,j,k,l}  = 1 - \frac{P_{ik} P_{jl} }{P_{il}P_{jk}} = \frac{P_{il}P_{jk} - P_{ik} P_{jl} }{P_{il}P_{jk}} = \frac{P_{ij}P_{lk} }{P_{il}P_{jk}} = -w_{i,k,j,l}.$$
		
		\item $$ \frac{w_{i,j,m,k} - 1}{w_{i,j,m,k} - w_{i,j,l,k}} = \frac{P_{mk}P_{jl}}{P_{jk}P_{ml}} = w_{m,j,k,l}.$$ 
		
		Here we use previous formula and (again) a Plucker identity in order to simplify denominator.

		\item $$w_{i,j,k,l}w_{i,j,k,m}^{-1}w_{i,j,l,m} = \frac{P_{ik} P_{jl} P_{im}P_{jk}P_{il}P_{jm}}{P_{il}P_{jk}P_{ik} P_{jm}P_{im}P_{jl}}= 1.$$
	\end{enumerate}	
	
\end{proof}	

\begin{rem}
	The name "cross-ratio" was chosen not by an accident. It should emphasize a deep connection between  $W/H$ and configurations of (distinct) points on projective line. More precisely, there is a bijection between  points of $W/H$ and configurations of (distinct) points on $\cp^1$. This is an example of Gelfand-MacPherson correspondence. From this point of view our "cross-ratios" become truly cross-ratios of points from configuration. The description of the Gelfand-MacPherson correspondence will be given in the next section.
\end{rem}

\subsection{Gelfand-MacPherson correspondence.}

Now we briefly describe Gelfand-MacPherson correspondence for arbitrary Grassmann manifolds. This correspondence connect points of $G_{q+1,k}$ and configuration of points on $\cp^{k-1}$. For more details see \cite{Kap} and \cite{GM}.

Denote by $\matkn$ the space of all comples $(q+1) \times k$ (matrices with $n$ rows and $k$ columns). One could notice that $\matkn$ possesses actions of two complex Lie groups: left action of $H = (\Complex^*)^{q+1}$  and right action of $GL(k; \Complex)$. There is an open set $\matknm$ of all matrices in $\matkn$. $\matknm$ consists of all matrices, such each $k \times k$ minor is non-zero. One could see that $\matknm$ is invariant under action of $H$ and $GL(k,\Complex)$. 
Moreover, we have next proposition.

\begin{prop}
	$\matknm / GL(k; \Complex)$ is the main stratum $W$ of $\grass$.
\end{prop}

\begin{rem}
	For an arbitrary $k$ the definiton of the main stratum and the space of parameters of stratum are the same as for $k=2$. See \cite{BT1}.
\end{rem}

The proof of this proposition is straightforward.

Now we can look at  $(\cp^{k-1})^n$. We can choose a subset of pairwise distinct points $(\cp^{k-1})^n_{max}: = \{(x_1, \dots, x_{q+1}) \in (\cp^{k-1})^n \lvert x_i \neq x_j  \}$. $(\cp^{k-1})^{q+1}_{max}$ has a very simple description.

\begin{prop}
	$(\cp^{k-1})^{q+1}_{max} = H \backslash \matknm$. Here $H \backslash \matknm$ denotes left action of $H$.
\end{prop} 

The proof is also straightforward and obvious.

Now we can combine two previous propositions and obtain an important corollary.

\begin{prop}
	$F = W / H = (\cp^{k-1})^{q+1}_{max} / PGL(k; \Complex).$
\end{prop}

This is exactly Gelfand-MacPherson correspondence.
One could observe that $(\cp^{k-1})^{q+1}_{max} / PGL(k; \Complex)$ is set of all configurations of pairwise distinct points up to automorphism of $\cp^{k-1}$. In the case $k=2$ we have that $F$ consists of all such configurations on $\cp^1$. This observation is very crucial for us. The observation is also explains the definiton of "cross-ratios". 

For the sake of completeness, we should mention the strongest version of this corresponedce.

\begin{thm}[\cite{Kap}, theorem $2.2.4$]
	The Gelfand-MacPherson correspondence extends to an isomorphism of Chow quotients $\grass // H$ and $(\cp^{k-1}) // PGL(k; \Complex)$.
	
\end{thm}

~\\

\subsection{Cross-ratios and torus action on $\gtwo$ .}

In this section we describe connection between torus action on $\gtwo$ and cross-ratios.

Fix some point $L \in W$. It's Plucker coordinates are $P_{ij}(L) \neq 0$ for an arbitrary $i,j \in \{1, \dots, n \}$. Denote orbit of $L$ under action of $H$ as $H.L$. The Plucker coordinates $P_{ij}(\tau.L)$ of the element $\tau.L \in H.L$ are equal $\tau_i \tau_j P_{ij}(L)$. 
One could see the next identity:

$$\frac{P_{ik}(\tau.L)P_{jl}(\tau.L)}{P_{il}(\tau.L)P_{jk}(\tau.L)}= \frac{\tau_i\tau_j\tau_k \tau_lP_{ik}(L)P_{jl}(L)}{\tau_i\tau_j\tau_k \tau_lP_{il}(L)P_{jk}(L)} = \frac{P_{ik}(L)P_{jl}(L)}{P_{il}(L)P_{jk}(L)}.$$

Now, for some $I = \{i,j,k,l\} \subset \{1,\dots ,n \}$ and for $L \in W$ denote $c'_I(L) := P_{ik}(L)P_{jl}(L)$ and $c_I(L) := P_{il}(L)P_{jk}(L)$. Now, we can write previous equality in the next form:

$$c'_I P_{ik}(\tau.L)P_{jl}(\tau.L) = c_I P_{il}(\tau.L)P_{jk}(\tau.L), $$

or $$ c'_I P_{ik}(\tau.L)P_{jl}(\tau.L) - c_I P_{il}(\tau.L)P_{jk}(\tau.L) = 0. $$

So, orbit of $L$ lies in the set of common zeroes of such polynomials.
Also, we should notice, that $c_I \neq c'_I$. Really, if $c_I = c'_I = c \neq 0$, we can see, that 

$$ P_{ik}(\tau.L)P_{jl}(\tau.L) - P_{il}(\tau.L)P_{jk}(\tau.L) = 0. $$

But by Plucker identities, we have, that 

$$ P_{ik}(\tau.L)P_{jl}(\tau.L) - P_{il}(\tau.L)P_{jk}(\tau.L) = P_{ij}(\tau.L)P_{kl}(\tau.L). $$

Hence, we have, that either $P_{ij}(\tau.L)=0$ or $P_{kl}(\tau.L) = 0$ and we have a contradiction that $L \in W$.

Now, we can prove the next proposition.

\begin{prop}
	For any $L \in W$ the orbit $H.L$ is defined by equations $$ c'_I P_{ik}(\tau.L)P_{jl}(\tau.L) - c_I P_{il}(\tau.L)P_{jk}(\tau.L) = 0. $$
	Here $c_I$ and $c'_I$ are arbitrary complex numbers, such $c_I \neq c'_I$.
\end{prop}

\begin{proof}[Proof]
One part of the proposition was proved above. So, we need to prove the next statement: any $M \in \gtwo$, which Plucker coordinates satisfy the equation 
$$ c'_I(L) P_{ik}(M)P_{jl}(M) - c_I(L) P_{il}(M)P_{jk}(M) = 0 $$ 
for some $L \in \gtwo$ is actually lies in $H.L$.

This is a simple computation:

$$P_{ik}(M) = \frac{c_I P_{il}(M)P_{jk}(M)}{c'_I P_{jl}(M)} = \frac{P_{ik}(L)P_{jl}(L)P_{il}(M)P_{jk}(M)}{P_{il}(L)P_{jk}(L) P_{jl}(M)}.$$

Let $\tilde{\tau_{ik}} := \frac{P_{ik}(M)}{P_{ik}(L)}$. This is easy to see that $\tilde{\tau_{ik}} = \frac{\tilde{\tau_{jk}} \tilde{\tau_{il}}}{\tilde{\tau_{jl}}}$. Hence,all $\tilde{\tau}_{ij}$-s  satisfy the equalities for torus "characters", as it was explained earlier. So, we have, that $$P_{ik}(M) = \tilde{\tau}_{ik}P_{ik}(L) ,$$ and $M$ is actually lies in $H.L$.  
\end{proof}

\subsection{Coordinates on the space of parameters of the main stratum.}

	Let $I:=\{i,j,k,l\} \subset \{1, \dots ,n\}, i<j<k<l$ and let $w_I = w_{ijkl}$ --cross-ratio. Each cross-ratio is a function on  $F$ with values in $\cp^1_A := \cp^1 \setminus \{0,1,\infty\}$. we may regard this map as map into $\cp^1$. Moreover, we can describe  this map via Plucker coordinates: $w_I(c) = [P_{ik}P_{jl}: P_{il}P_{jk}]$, $c \in F$. So, we can define a map $\Phi: F \rightarrow (\cp^1)^N, ~N=\binom{q+1}{4}$ by the formula:

$$w(c) = (w_{1234}(c); \dots ;w_{n-3, n-2, n-1, n}(c)).$$

 Later we will show that $\Phi$ is an embedding. One can see, that $\Phi(F)$ actually lies in some subvariety $X$ in $(\cp^1)^N$, because cross-ratios satisfy some identities.

Let $F$ -- space of parameters of the main stratum in $\gtwo$. Denote $z_l: = w_{123l}$, где $l = 4, \dots, q+1$. 

\begin{prop}
	The functions $z_4, \dots, z_q+1$ are coordinates on $F$ .
\end{prop}

\begin{proof}[Proof]
	Via Gelfand-MacPherson correspondence for each $c \in F$ (i.e. orbit of element from $W$) we can construct a configuration of $n$ pairwise distinct points on $\cp^1$ (up to action of $PGL(2, \Complex)$). It's a classical result that any $4$ points on $\cp^1$ are uniquely (again, up to $PGL(2, \Complex)$ action) defined by cross-ratio. Hence, if we know the numbers $z_4, \dots, z_n$, then (under an assumption that first 3 points are $[1:0],[0:1]$ and $[1:1]$) we can rebuild a configuration of points (up to $PGL(2, \Complex)$) and hence  $c$. 
\end{proof}

\begin{rem}
	Notice, that $z_i \neq 0$ and $z_i \neq 1$, because we are living on $F$.
\end{rem}

\begin{prop}
	If $k \neq l$, then $z_k \neq z_l$.
\end{prop}

\begin{proof}[Proof]
	Suppose $z_k = z_l$. Then $P_{2k}P_{1l} = P_{1k}P_{2l}$. By Plucker identities $P_{12}P_{kl} - P_{1k}P_{2l} + P_{1l}P_{2k} = 0$, and hence $P_{12}P_{kl} = 0$. We have a contradiction with the fact that we are "living" on the main stratum $W$.
\end{proof}

Now we can write down all other cross-ratios via $z_k$.

\begin{thm}
	Cross-ratios $w_{123i}$ define a coordinates on $F$. All other cross-ratios defined by the next formulas:
	
	\begin{enumerate}
	\item$$w_{12ij} = \frac{z_j}{z_i}, $$ 
	\item$$w_{13ij} = \frac{1 - z_j}{1 - z_i},$$ 
	\item$$ ~w_{23ij} = \frac{z_i(1 - z_j)}{z_j(1 - z_i)},  $$
	
	\item$$w_{1ijk} = \frac{z_i - z_k}{z_i - z_j},$$ 
	\item$$ ~w_{2ijk} = \frac{z_j (z_i - z_k)}{z_k(z_i - z_j)},$$ \item$$~w_{3ijk} = \frac{(1 - z_j)(z_i - z_k)}{(1 - z_k)(z_i - z_j)} ,$$
	
	\item$$w_{ijkl} = \frac{(z_i - z_k)(z_j - z_l)}{(z_i - z_l)(z_j - z_k)}. $$ Here $i, j, k, l \in \{4, \dots , q+1\}$ and $i<j<k<l$.
	
	\end{enumerate}
		
\end{thm}

\begin{proof}[Proof] The proof of these identities is nothing but long computations by using symmetries of $w_{i,j,k,l}$ and the formula $w_{i,j,k,l}w_{i,j,k,m}^{-1}w_{i,j,l,m} = 1$.
	\begin{enumerate}
		\item By the formulas for cross-ratios we have such identity:
		 $$w_{1,2,i,j}w_{1,2,i,3}^{-1}w_{1,2,j,3}= w_{1,2,i,j}w_{1,2,3,i}w_{1,2,3,j}^{-1} = 1$$
		Hence we have:
		$$w_{1,2,i,j} = \frac{w_{1,2,,3,j}}{w_{1,2,3,i}}=\frac{z_j}{z_i}. $$
		
		 \item Here is analogy formula: $$w_{1,3,i,j}w_{1,3,i,2}^{-1}w_{1,3,j,2}= w_{1,2,i,j}w_{1,3,2,i}w_{1,3,2,j}^{-1} = 1$$
		
		Since $w_{1,3,2,i} = w_{1,2,3,i} - 1 = z_i - 1$, we have the next formula:
		
		$$w_{1,3,i,j} = \frac{w_{1,3,2,j}}{w_{1,3,2,i}} = \frac{1-z_j}{1-z_i}. $$
		
		\item Because $w_{2,3,i,j} = w_{i,j,2.3}$ and $$w_{i,j,2,3} w^{-1}_{i,j,2,1} w_{i,j,3,1} =1$$ we have the formula:
		
		$$w_{2,3,i,j} = \frac{z_i(1-z_j)}{z_j(1-z_i)}. $$

		\item 
		$$w_{1ijk}w^{-1}_{1ij2}w_{1ik2}=1. $$
		For any $i,j$ we have the formula: $w_{1,i,j,2} = \frac{1}{w_{1,i,2,j}} = \frac{1}{w_{1,2, i,j} - 1}.$
		Substitute it in previous formula, we have identity:
		
		$$w_{1ijk} = \frac{w_{1,2,i,k} - 1}{w_{1,2,i,j} - 1} = \frac{z_k - z_i}{z_j - z_i}. $$
		
		\item 
		We have two formulas:
		
		$$w_{2,i,j,k}w^{-1}_{2,i,j,1}w_{2,i,k,1} =1 $$
		
		and
		
		$$w_{2,i,j,1} = \frac{w_{1,2,i,j}}{1 - w_{1,2,i,j}} = \frac{z_j }{(z_i - z_j)}. $$
		
		Combining it, we obtain the formula:
		
		$$w_{2,i,j,k} = \frac{z_j(z_i-z_k)}{z_k(z_i -z_j)}. $$
		
		\item $$w_{3,i,j,k}w^{-1}_{3,i,j,1}w_{3,i,k,1} =1,$$
		
		hence 
		
		$$w_{3,i,j,k} = \frac{w_{3,i,j,1}}{w_{3,i,k,1}}. $$
		
		By formula $w_{3,i,j,1} = \frac{w_{1,3,i,j}}{1 - w_{1,3,i,j}} = \frac{1 - z_j }{(z_j - z_i)}$  we obtain follows:
		
		$$w_{3,i,j,k} = \frac{(1- z_j)(z_i - z_k)}{(1-z_k)(z_i - z_j)}. $$
		
		\item 	We have an identity $$w_{i,j,k,l}w_{i,j,k,1}^{-1}w_{i,j,l,1} = 1.$$	
		
		We also have $w_{i,j,k,1} = w_{1,k,j,i}$ and hence $$w_{i,j,k,l} = \frac{w_{1,k,j,i}}{w_{1,l,j,i}} = \frac{(z_k - z_i)(z_l - z_j)}{(z_k - z_j)(z_l - z_i)}. $$

	\end{enumerate}
\end{proof}

\begin{rem}
Maybe it's rather straightforward than any other possible proofs. One could compute directly via Plucker coordinates or by using standard formulas for cross-ratio (or create absolutely different proof). 
\end{rem}

Now we can properly describe $F$ in $\Complex^{q-2}$.

\begin{prop}
	$F$ defines in $\Complex^{q-2}$ by next inequalities: $z_i \neq 0$, $z_i \neq 1$ and $z_i \neq z_j, ~i,j =4, dots, q+1$.
\end{prop}

\begin{proof}[Proof]
	We know that all $w_{i,j,k,l} \neq 0, 1$. From previous theorem we know, that $w_{i,j,k,l} = 0$ iff some $z_i = 0$ or 1 or $z_a - z_b = 0$ for some indices a,b. Similary, $w_{i,j,k,l} = 0$ iff $(z_k - z_l)(z_i - z_j) = 0 $ (for $i,j,k,l \geq 4$). For other cases proof is the same.
\end{proof}

\begin{rem}
One may use another set of cross-ratios (instead of $w_{1,2,3,i}$) if they satisfies the same properties as $w_{1,2,3,i}$. Moreover, if we have such set of cross-ratios, we can write down the same formulas even if \emph{some} cross-ratios are equal 0,1 or $\infty$. This observation will be useful in the next sections.
\end{rem}

\subsection{A manifold in $(\cp^1)^N$ , $N = \binom{q+1}{4}$ defined by cross-ratios.}

We showed above, that cross-ratios satisfies some identities. These identities are defines some variety in $(\cp^1)^{\binom{q+1}{4}}$. A priori this variety may be non smooth. However, there is a nice description of this variety.

\begin{thm}[See also \cite{MS}, Appendix D]
	The closure of the image of $F$ under the map $\Phi$ is a smooth complex manifold of complex dimension $q-2$. This closure coincide with the variety $\mathcal{W}_{q+1}$, which is determinated by identities for cross-ratios. 
\end{thm}

\begin{rem}
	This variety is actually a moduli space $\moduli$ of genus zero stable curves with n marked points. 
\end{rem}

\subsection{Embedding of other spaces of parameters in the products of $\cp^1$.}

Here we describe how one could construct the embedding of any $F_\sigma$ for an arbitrary stratum $W_\sigma$. 

As before, we start from the fact that action of $H$ preserves cross-ratios. But now, since we are living not in the main stratum, we can't define all cross-ratios. Really, for stratum $W_\sigma$ there is some Plucker coordinate $P_I = 0$ on this stratum. Hence, there are cross-ratios, which are not well-defined. Instead it we have some restrictions on Plucker coordinates.

But if we can define cross-ratio $w_{i,j,k,l} = \frac{P{ik}P_{jl}}{P_{il}P_{jk}}$ then it's automatically invariant under action of $H$ (and hence, under action of $T^n$). As we showed before, this is also sufficient: two points $L_1, L_2 \in W_\sigma$ lies in the same orbit of $H$ iff $w_{i,j,k,l}(L_1) = w_{i,j,k,l}(L_2)$ for all well-defined cross-ratios (the proof is absolutely the same as in the case of $W$).

We also have other case. Some cross-ratio may be equal 1, i.e. $w_{i,j,k,l}(L)=1$ for some $L \in W_\sigma$. But we have:

$$w_{i,j,k,l} = 1 \iff P_{ik}P_{jl} = P_{il}P_{jk} \iff P_{ik}P_{jl} - P_{il}P_{jk} = 0.  $$

But by Plucker identities we have, that:

$$P_{ik}P_{jl} - P_{il}P_{jk} = P_{ij}P_{kl}. $$

Hence such situation appears if either $P_{ij} = 0$ or $P_{kl}=0$.

Denote by $K_\sigma$ the number of all well-defined cross-ratios on $W_\sigma$. Now we can construct the map $\Phi_\sigma: F_\sigma \rightarrow (\cp^1)^{K_\sigma}$ by the formula $\Phi_\sigma(c) = (w_{I_1}; \dots, w_{I_{K_\sigma}})$. Here by $I_1, \dots,I_{K_\sigma} $ we denotes all 4-tulpes of indices, such as $w_{I_j}$ is well-defined cross-ratio on $W_\sigma$.

\begin{defn}
	\begin{enumerate}
		\item
		We will say a cross-ratio $w_{i,j,k,l}$ is a \emph{strongly admissible} (with respect to $W_\sigma$), if for each point $p \in W_\sigma$ $w_{i,j,k,l}(p)$ is finite and $w_{i,j,k,l}(p) \neq 0,1$. The such 4-tuple $\{i,j,k,l\}$ is called a strongly admissible tuple. We denote $\mathcal{I}_{s, \sigma}$ the set of all such 4-tuples.
		
		\item If $w_{i,j,k,l}(p)$ is equal either $0, 1$ or $\infty$ then we will call it \emph{weakly admissible}. The corresponding 4-tuple is called weakly admissible and $\mathcal{I}_{w, \sigma}$ denotes the set of all weakly admisslble 4-tuples.
		
		\item  If $w_{i,j,k,l}$ is not definite on $W_\sigma$, then we will call it \emph{non-admissible}. $\mathcal{I}_{n, \sigma}$ denotes the set of all 4-tuples $\{i,j,k,l\}$, such $w_{i,j,k,l}$ is non-admissible.
	\end{enumerate}

We want to emphasize that all these definitions depends on the stratum.
\end{defn}

Easy to see, that a property of cross-ratio be (non-)admissible does not changes under permutations of indices.

\begin{rem}
	Easy to see, that if all cross-ratios are admissible on $W_\sigma$, then one could construct the embedding of $F_\sigma$ into $(\cp^1)^{\binom{q+1}{4}}$ and the image of this embedding lies into the closure $\overline{F}$ of the main stratum.
\end{rem}

\section{The main theorem.}

Our goal is proving the next theorem.

\begin{thm}\label{TM}
The universal space of parameters for $\gtwo$ is $\mathcal{F}=\overline{F} = \gtwo // H$.
\end{thm}

We want to emphasize, that universal space of parameters may be constructed only from the space of parameters of the main stratum, i.e. without any information from other strata.

In order to prove it we need some additional constructions. These constructions are needed for cheching all condition in Axiom 6 of universal space of parameters. First of all, we should define $\tilde{F}_\sigma$ and construct the projections from it onto $F_\sigma$. And the hardest part, we should show that maps $\mathcal{H}: \cup_\sigma \mathring{P}_\sigma \times \tilde{F}_\sigma \rightarrow \Delta_{q+1,2} \times \overline{F}$ from Axiom 6 is continious. 

\subsection{Subsets $\tilde{F}_\sigma$ in $\overline{F}$.}

Recall that $\cp^1_A = \cp^1 \setminus \{0,1, \infty\}$. 
\begin{defn} \label{defn11}
	Fix a stratum $W_\sigma$. We define $\tilde{F}_\sigma \subset \overline{F}$ by the following way:
	
	$\tilde{F}_\sigma = \{x\in \overline{F} ~|~ w_I(x) \in \cp^1_A \iff I \in \mathcal{I}_{s, \sigma}; w_J(x) \in \{0,1,\infty \} \iff \mathcal{I}_{w, \sigma}\}$.
\end{defn}

Notice, that there is a projection $g_\sigma : \tilde{F}_\sigma \rightarrow F_\sigma$, defined by formula $g_\sigma(w_{1,2,3,4}, ~\\ \dots, w_{q-2,q-1,q,q+1}) = (w_{I_1}; \dots, w_{I_{K_\sigma}})$
-- projection to the set of all strongly admissible cross-ratios of stratum $W_\sigma$. This projection is obviously surjective and continious. 

\begin{exampl}[Examples]
	\begin{enumerate}
		\item For the main stratum $\tilde{F} = F$;
		
		\item For any fixed point, i.e. for $\sigma= \{ij\}$ $\tilde{F}_\sigma = \mathcal{F}$. 
	\end{enumerate}
\end{exampl}

\subsection{Proof of the Theorem \ref{TM}.}

In order to prove this theorem \ref{TM} we need to check all conditions in Axiom \ref{A6}:

	\begin{enumerate}
		\item For the main stratum $W$ there is an equality $\tilde{F} = F$. Here $F$ -- space of parameters of the main stratum;
		
		\item $\mathcal{F}$ is a compactification of $F$;
		
		\item If $P_{\sigma_1}$ lies in $\partial P_\sigma$, then $\tilde{F}_\sigma  \subset \tilde{F}_{\sigma_1}$;
		
		\item $\mathcal{F} = \cup_\sigma \tilde{F}_\sigma$;
		
		\item For all $\sigma$ there exist $p_\sigma: \tilde{F}_\sigma \rightarrow F_\sigma$;
		
		\item The map $\mathcal{H}: \cup_\sigma \mathring{P}_\sigma \times \tilde{F}_\sigma \rightarrow \Delta_{q+1,2} \times \mathcal{F}$, defined as $\mathcal{H}(x,c) = (x; p_\sigma(c))$ is continious map.
	\end{enumerate}

\begin{proof}[Proof]
	
	\begin{enumerate}
		\item This is obviously follows from the construction of $\mathcal{F}$;
		
		\item  Same as first condition;
		
		\item Due to corollary of theorem \ref{ThmSix}, it's enough to proof this for $q$-dimensional polytopes in $\hprsmplx$.

		First of all, suppose that $P_{\sigma_1} \subset \partial \hprsmplx$ is a hypersimplex itself, which lies in the boundary of admissible $q$-dimensional polytope $P_\sigma$. Then it consists a planes, which lies in some coordinate hyperplane $L_i \subset \Complex^{q+1}$. Hence, for any $j$ all $P_{ij} = 0$ and for any 4-tuple $I \ni i$ the corresponding cross-ratio $w_I$ becomes non-admissible. All other cross-ratios remain. By definition of $\tilde{F}_\sigma$ we have an embedding $\tilde{F}_\sigma$ into $\tilde{F}_{\sigma_1}$ by obvious way, since any point from $\tilde{F}_\sigma$ also belongs to $\tilde{F}_{\sigma_1}$.
		
		Suppose now that $P_{\sigma_1}$ is a product of two simplices (as in theorem \ref{ThmSix}), which lies in $\partial P_\sigma$ and $\dim P_\sigma =q$ as before. Easy to see that there is no strongly admissible cross-ratios for $W_{\sigma_1}$ 
		
		\begin{prop}
			There are no strongly admissible cross-ratios for $W_{\sigma_1}$. The cross-ratio $w_{i,j,k,l}$ is weakly admissible $\iff$ $i \sim j$ and $k \sim l$. In this case $w_{i,j,k,l} = 1$
		\end{prop}
		
		\begin{proof}[Proof of proposition]
			Let $w_{i,j,k,l}$ be an cross-ratio. There is a two equivalence classes on $[q+1]$ , which corresponded to the stratum $W_{\sigma_1}$ (see theorem \ref{ThmSix}). 
			Suppose that $i \nsim j$ and $j \sim k,l$. Then $w_{i,j,k,l}$ is non-admissible. If $i \sim k$ and $j \sim l$ then $w_{i,j,k,l}$ again is non-admissible. In the end, if $i \sim j$ and $k \sim l$, then $w_{i,j,k,l} = \frac{P_{ik}P_{jl}}{P_{il}P_{jk}} =1$ by Plucker identity 
			
			$$P_{ij}P_{kl} - P_{ik}P_{jl} + P_{jk}P_{il}=0.$$
			
			By symmetries of cross-ratios, all other case might be reduced to previous cases.
			\end{proof}
		
		The only problem might with embedding $\tilde{F}_\sigma  \subset \tilde{F}_{\sigma_1}$ in our case might be follows: the cross-ratio $w_{i,j,k,l}$ might be weakly admissible for $W_{\sigma_1}$ and strongly adissible for $W_{\sigma}$. Now we are going to show that it can't be happen. 
		
		If $P_{\sigma_1}$ is defined by an equation $q_I(x) =\sum_{i \in I} x_i = 1$ (see theorem \ref{ThmSix}), then $P_{\sigma}$ is defined by either $\sum_{i \in I} x_i \geq 1$ or $\sum_{i \in I} x_i \leq 1$.	Suppose that $w_{i,j,k,l}$ is weakly admissible for $W_{\sigma_1}$. In that case $P_{ij} = P_{kl} = 0$. If $w_{i,j,k,l}$ is strongly admissible for $W_{\sigma}$, then both $P_{ij}$ and $P_{kl}$ are non-zero. Hence, both $e_{ij}$ and $e_{kl}$ are verticles of $P_\sigma$.	But $q_I(e_{ij})$ and $q_I(e_{kl})$ have different signes. Hence, in this case $P_\sigma = \hprsmplx$ and $P_{\sigma_1}$ does not lie in $\partial P_\sigma$. Contradiction.
		
		Hence, if cross-ratio $w_{i,j,k,l}$ is weakly admissible for $W_{\sigma_1}$, it's also weakly adissible for $W_{\sigma}$. Thus, the embedding $\tilde{F}_\sigma  \subset \tilde{F}_{\sigma_1}$ defined correctly, because any point from $\tilde{F}_\sigma$ lies in $\tilde{F}_{\sigma_1}$.
		
		Now, by induction we obtain this result for any stratum.
		
		\item Directly follows from previous. 
		
		\item As we mentioned before, for any $\tilde{F}_\sigma$ there is a continious and surjective projection $g_\sigma : \tilde{F}_\sigma \rightarrow F_\sigma$, defined by formula $g_\sigma(w_{1,2,3,4}, ~\\ \dots, w_{n-3,n-2,n-1,n}) = (w_{I_1}; \dots, w_{I_{K_\sigma}})$
		
		\item 	Set $\mathcal{P}: = \cup_\sigma P_\sigma$ and $\mathcal{E} := \cup_\sigma \mathring{P}_\sigma \times \tilde{F}_\sigma \subset \mathcal{P} \times \mathcal{F}$.
		
		Define $\tilde{p}: \mathcal{P} \rightarrow \Delta_{q+1,2}$ -- obvious canonical projection. Also, let $\tilde{\mu} = \tilde{p} \circ \mu$.

		Now, fix the stratum $W_\sigma \in \gtwo$. Let $\{p_m \} \subset W$ be a sequence in the main stratum $W$ and $\{q_m\} \subset W_\sigma$ be a sequence in $W_\sigma$. Suppose that these both sequenses converges to the point $p \in W_\sigma$. Easy to see, that each Plucker coordinate $P_{ij}(p_m) \rightarrow P_{ij}(p)$ and $P_{ij}(q_m) \rightarrow P_{ij}(p)$ as $m$ tends to infinity. Since each $p_m$ lies in some orbit of complex torus $H$, we have a corresponding sequense  $\{c_m\} \subset F = W/H$. From the description of the embedding of $F$ into $(\cp^1)^N$, that sequence $\{c_m\}$ is a Cauchy sequense. Hence it has a limit in $\overline{F}\subset (\cp^1)^N $. We also have a similar sequence $\{d_m\} \subset F_\sigma \subset (\cp^1)^{K_\sigma}$.It also has a limit $d \in F_\sigma$ (since the limit of $\{q_m\}$ is $p \in F_\sigma$). 
		
		Now, from the definitions of $\tilde{F}_\sigma$ and projections $g_\sigma: \tilde{F}_\sigma \rightarrow F_\sigma$ one can see that $\lim\limits_{m \rightarrow \infty} g_\sigma (s(h (p_m)) = \lim\limits_{m \rightarrow \infty} s_\sigma(h_\sigma(q_m)) = s_\sigma(h_\sigma(p))$. Also, from the definition of moment map for Grassmanian, one could see that $\lim\limits_{m \rightarrow \infty}\tilde{\mu}(p_m) = \lim\limits_{m \rightarrow \infty} \tilde{\mu}(q_m)$.
		Hence we have that $H$ is continious map as composition of continious maps. So, all axioms of universal space of parameters are accomplished and, by this, $\tilde{F}$ is an universal space of parameters for $\gtwo$.	
		
	\end{enumerate}
\end{proof}

\begin{cor}
	Universal space of parameters $\overline{F}$ is Chow factor of $G_{q+1,2}$ by $H$. 
\end{cor}

\begin{proof}[Proof]
	By the theorem of Salamon and McDuff \cite{MS}, $\overline{F} = \moduli$ -- the moduli space of stable curves of genus zero witn $n$ marked points. But by the theorem of Kapranov \cite{Kap} , $G_{q+1,2} // H = \moduli$.
\end{proof}

\section{Examples for small q.}

In this section we will check the cases of $n=4$ and $n=5$.

\subsection{Example for q=3.}

For the case $G_{4,2}$ we have $\dim_{\Complex}\mathcal{F}_{4,2} = 1$ and $\binom{4}{4} = 1$. So, we have only one cross-ratio $w_{1,2,3,4}$. The space of parameters of the main stratum $F$ is $\cp^1_A = \cp^1 \setminus \{0,1,\infty \}$. Hence, the closure $F$ in $\cp^1$ is whole $\cp^1$. In this case, for each stratum all virtual spaces of parameters are the same as (ordinary) spaces of parameters. This result is the same as in \cite{BT3}.

\subsection{Example for q=4.}
In \cite{BT1} Buchstaber and Terzic described stratification of $G_{5,2}$ constructed the universal space of parameters for it. They showed that $\mathcal{F} = \cp^2 \sharp 4 \overline{\cp}^2$ and Proposition 28 from their paper implies that their universal space of parameters is the same as our in the case $q=4$ (it's also following from examples from \cite{MS}). Here we demonstrate this case and describe all virtual spaces of parameters in this case (note that our definition of virtual spaces of parameters is differs from definition of Buchstaber and Terzic). 

 We have $\binom{5}{4} = 5$ cross-ratios and $\dim_{\Complex}\mathcal{F}_{5,2} = 2$. Let $w_{1,2,3,4} = \frac{c_1}{c'_1}, w_{1,2,3,5} = \frac{c_2}{c'_2}, w_{1,2,4,5} = \frac{c_3}{c'_3}, w_{1,3,4,5}=\frac{c_4}{c'_4}$ and $w_{2,3,4,5} = \frac{c_5}{c'_5}$. 

By the formulas for cross-ratios, we have four equations for $c_i,c'_i$:

\begin{enumerate}
	\item $$c_1 c'_2 c_3 = c'_1 c_2 c'_3; $$
	
	\item $$c_4(c'_1 - c_1)c'_2 = c'_4 (c'_2 - c_2)c'_1 ; $$
	
	\item $$c_5(c'_1 - c_1) c_2 = c'_5 (c'_2 - c_2)c_1 ;$$
	
	\item $$c_5 c'_4 c_3 = c'_5 c_4 c'_3. $$
\end{enumerate}

In order to describe virtual spaces of parameters for each strata we need know, how  all strata looks like. The description of each stratum was done in \cite{BT1} (section 4.3) . We also change notations in this case, because general notation from Section \ref{sect12} is not convenient in this case. 

\begin{defn}
	Let $\mathcal{I} = \{I_1,\dots ,I_k\}$ be a subset of the set $K = \{I \subset 2^{\{1,\dots,n  \}  } | ~|I|=2 \}$. We denote $W_{\mathcal{I}} = \{ L \in G_{5,2} ~|~ P_I(L) = 0, ~I\in \mathcal{I}   \} $. We also denote the corresponding virtual space of parameters by  $\tilde{F}_{\mathcal{I}}$.
\end{defn}

For exapmle, $W_{ij} = \{ L \in G_{5,2} ~|~ P_{ij}(L) = 0 \}$
and $W_{ij, kl} = \{ L \in G_{5,2} ~|~ P_{ij}(L)= P_{kl}(L) = 0 \}$.

Now we want to describe the universal spaces of parameters in the sense of Definition \ref{defn11}.

\begin{prop} For any stratum $W_{ij}$ the corresponding virtual spaces (according to definition \ref{defn11}) of parameters are following: 
	\begin{enumerate}
		\item $\tilde{F}_{12} = ([1:1];[1:1];[1:1];[c:c'];[c':c]);$
		\item $\tilde{F}_{13} = ([0:1];[0:1];[c:c'];[1:1];[c':c]);$
		\item $\tilde{F}_{14} = ([1:0];[c:c'];[0:1];[0:1];[c-c':c]);$
		\item $\tilde{F}_{15} = ([c:c'];[1:0];[1:0];[1:0],[c:c-c']);$
		\item $\tilde{F}_{23} = ([1:0];[1:0];[c:c'];[c:c'];[1:1]);$ 
		\item $\tilde{F}_{24} = ([0:1];[c:c'];[1:0];[c'-c:c'];[0:1]);$
		\item $\tilde{F}_{25} = ([c:c'];[0:1];[0:1];[c':c'-c];[1:0]);$
		\item $\tilde{F}_{34} = ([1:1];[c:c'];[c:c'];[1:0];[1:0]);$
		\item $\tilde{F}_{35} = ([c:c'];[1:1];[c':c];[0:1];[0:1]);$
		\item $\tilde{F}_{45} = ([c:c'];[c:c'];[1:1];[1:1];[1:1]);$
	\end{enumerate}

Here $[c:c'] \in \cp^1_A = \cp^1 \setminus \{0;1; \infty\}$.

\end{prop}

\begin{proof}[Proof]
	There is no non-admissible cross-ratios for any $W_{ij}$, hence the ordinary spaces of parameters coincide with virtual spaces of parameters. 
\end{proof}

\begin{rem}
	One could notice that the list above is coincide with list in Lemma 25 from \cite{BT1}.
\end{rem}

\begin{prop} For any stratum $W_{ij,kl}$ the universal space of parameters are following
	\begin{enumerate}
		\item $\tilde{F}_{12,34} = ([1:1];[1:1];[1:1];[1:0];[1:0]);$
		\item $\tilde{F}_{12,35} = ([1:1];[1:1];[1:1];[0:1];[0:1]);$
		\item $\tilde{F}_{12,45} = ([1:1];[1:1];[1:1];[1:1];[1:1]);$
		\item $\tilde{F}_{13,24} = ([0:1];[0:1];[1:0];[1:1];[0:1]);$
		\item $\tilde{F}_{13,25} = ([0:1];[0:1];[0:1];[1:1];[1:0]);$
		\item $\tilde{F}_{13,45} = ([0:1];[0:1];[1:1];[1:1];[1:1]);$
		\item $\tilde{F}_{14,23} = ([1:0];[1:0];[0:1];[0:1];[1:1]);$
		\item $\tilde{F}_{14,35} = ([1:0];[1:1];[0:1];[0:1];[0:1]);$
		\item $\tilde{F}_{14,25} = ([1:0];[0:1];[0:1];[0:1];[1:0]);$
		\item $\tilde{F}_{15,24} = ([0:1];[1:0];[1:0];[1:0],[0:1]);$
		\item $\tilde{F}_{15,34} = ([1:1];[1:0];[1:0];[1:0],[1:0]);$
		\item $\tilde{F}_{15,23} = ([1:0];[1:0];[1:0];[1:0],[1:1]);$
		\item $\tilde{F}_{23,45} = ([1:0];[1:0];[1:1];[1:1];[1:1]);$
		\item $\tilde{F}_{24,35} = ([0:1];[1:1];[0:1];[0:1];[0:1]);$
		\item $\tilde{F}_{25,34} = ([1:1];[0:1];[0:1];[1:0];[1:0]);$
		  
	\end{enumerate}
\end{prop}

\begin{proof}[Proof]
	Fix a stratum $W_{ij,kl}$. One could check that there is no non-admissible cross-ratios for this stratum. So, virtual spaces of parameters coincide with ordinary spaces of parameters.
\end{proof}

Now we want to describe virtual spaces of parameters for stratum of type $W_{ij,kl,pq}$. According to \cite{BT1}, all such strata are actually $W_{ij,ik,jk}$.

\begin{prop}		For any stratum $W_{ij,ik,jk}$ the universal space of parameters are following
	\begin{enumerate}
		\item $\tilde{F}_{34,35,45} = ([1:1];[1:1];[1:1];[c:c'];[c:c']);$
		\item $\tilde{F}_{24,25,45} = ([0:1];[0:1];[c:c'];[1:1];[c':c]);$
		\item $\tilde{F}_{23,25,35} = ([1:0];[c:c'];[0:1];[0:1];[c-c':c]);$
		\item $\tilde{F}_{23,24,34} = ([c:c'];[1:0];[1:0];[1:0];[c:c-c']);$
		\item $\tilde{F}_{14,15,45} = ([0:1];[0:1];[c:c'];[c:c'];[1:1]);$
		\item $\tilde{F}_{13,15,35} = ([0:1];[c:c'];[1:0];[c'-c:c'];[0:1]);$
		\item $\tilde{F}_{13,14,34} = ([c:c'];[0:1];[0:1];[c':c'-c];[1:0]);$
		\item $\tilde{F}_{12,15,25} = ([1:1];[c:c'];[c:c'];[1:0];[1:0]);$
		\item $\tilde{F}_{12,14,24} = ([c:c'];[1:1];[c':c];[0:1];[0:1]);$
		\item $\tilde{F}_{12,13,23} = ([c:c'];[c:c'];[1:1];[1:1];[1:1]);$
	\end{enumerate}

Here $[c:c'] \in \cp^1$.
\end{prop}

\begin{proof}[Proof]
	For any stratum $W_{ij,ik,jk}$ there are exactly two non-admissible cross-ratio: $w_{i,j,k,p}$ and $w_{i,j,k,q}$. Any other cross-ratio is weakly admissible. So, by the equation $w_{i,j,k,p}w^{-1}_{i,j,k,q}w_{i,j,p,q} = 1$ we can find the relation between non-admissible cross-ratios.
	
	For example, we describe the case of $W_{34,35,45}$. Easy to see, that $w_{1,2,3,4}, w_{1,2,3,5}, w_{1,2,4,5}$ are weakly admissible and they all equal $1$. Moreover, $w_{1,3,4,5}$ and $w_{2,3,4,5}$ are both non-admissible. By the equation $ w_{1,3,4,5} w^{-1}_{2,3,4,5} w_{1,2,4,5} = 1$ (or, equivalently, from equation $c_5 c'_4 c_3 = c'_5 c_4 c'_3$ )we obtain the equality $w_{1,3,4,5} = w_{2,3,4,5} = \frac{c}{c'}$. Hence $\tilde{F}_{34,35,45} = ([1:1];[1:1];[1:1];[c:c'];[c:c'])$.
	
	All other cases are obtained in the same way and we omit it.
\end{proof}

We described all virtual spaces of parameters for non-main stratum, which corresponds to 4-dimensional polytopes in $\Delta_{5,2}$. Now we turn to 3-dimensional polytopes. First of all we want to deal with strata in $G_{4,2} \subset G_{5,2}$, which may be obtained via coordinate inclusions $\Complex^4 \rightarrow \Complex^5$.

\begin{prop}
	Denote $W_i := W_{ij,ik,il, im}$, where $\{i,j,k,l,m\} = \{1,2,3,4,5\}$. Denote corresponding virtual space of parameters as $\tilde{F}_i$. Then we have $\tilde{F}_i = F \bigcup (\cup_{j \neq i} \tilde{F}_{ij})$. 
\end{prop}

\begin{proof}[Proof]
	We prove it only for $W_1$. All other cases may be obtained from this case by action of permutation group $S_5$ on $G_{5,2}$.

	On $W_1$ all cross-ratios $w_{1,j,k,l}$ are non-admissible since all $P_{1j} = 0$ for any $j$. Hence only $w_{2,3,4,5}$ is admissible. Actually it is strongly admissible, because all Plucker coordinates $P_{kl}$  are non-zero if $k \neq 1$ and $l \neq 1$. Hence $\tilde{F}_1 = ([c_1:c'_1];[c_2:c'_2];[c_3:c'_3];[c_4:c'_4];[c_5:c'_5])$ and $[c_1:c'_1] \neq [1:0], [0:1],[1:1]$. Easy to see that $F \subset \tilde{F}_i$ and $\tilde{F}_{1j} \subset \tilde{F}_1$. According to \ref{defn11}, any point in $\tilde{F}_i$ has the form of $[c_1:c'_1];[c_2:c'_2];[c_3:c'_3];[c_4:c'_4];[c_5:c'_5]$ and $[c_5:c'_5] \in \cp^1_A$. But any such point lies in $F \bigcup (\cup_{j \neq 1} \tilde{F}_{1j})$. Hence, $\tilde{F}_i = F \bigcup (\cup_{j \neq i} \tilde{F}_{ij})$.
	
\end{proof}

Let $\Complex_i = \{z\in \Complex^5 ~|~ z_i=0 \}$. Denote by $Y_i$ stratum $W_{jk,jl,jm,kl,km,lm}$. This stratum consists of $L \in G_{5,2}$, such as $L \cap \Complex_i$ is a line. This is an open dense subset in $\cp^3 \subset G_{5,2}$.

\begin{prop}
	For any $Y_i$ the corresponding universal space of parameters is whole $\mathcal{F}_{5,2}$. 
\end{prop}

\begin{proof}[Proof]
	Obvious, since there is no admissible cross-ratios for $Y_i$.
\end{proof}

Now we want to study the case of strata, which corresponding polytopes are $3$-dimensional and it does not lie on $\partial \Delta_{5,2}$. According to Theorem \ref{ThmSix}, all such polytopes may be described via equation $x_i + x_j = 1$ in $\Real^5$. More precisely, suppose $\{k,l,m\} = \{1,2,3,4,5\} \setminus \{i,j\}$. Then the corresponding stratum is $W_{ij,kl,km,lm}$.

\begin{prop}
	For any $W_{ij,kl,km,lm}$ the corresponding virtual spece of parameters $\tilde{F}_{ij,kl,km,lm}$ is equal $\tilde{F}_{ij}$ (and is also equal $\tilde{F}_{kl,km,lm}$).
\end{prop}

\begin{proof}[Proof]
	For each stratum $W_{ij,kl,km,lm}$ we have, that $w_{i,k,l,m}$ and $w_{j,k,l,m}$ are non-admissible. We should show, that all other cross-ratios are admissible.
	
	In order to do this, we need to write down other 3 cross-ratios via Plucker coordinates:
	
	$$w_{i,j,k,l} = \frac{P_{ik}P_{jl}}{P_{il}P_{jk}},$$
	$$w_{i,j,k,m} = \frac{P_{ik}P_{jm}}{P_{im}P_{jk}},$$
	$$w_{i,j,l,m} = \frac{P_{il}P_{jm}}{P_{im}P_{jl}}.$$
	
	Since all Plucker coordinates in these formulas are not zero, we have that $w_{i,j,k,l},w_{i,j,k,m}$ and $w_{i,j,l,m}$ are admissible. This proves the proposition.
\end{proof}

Now we describe universal spaces of parameters for strata, which corresponds to admissible polytopes in $\Delta_{4,2} \subset \Delta_{5,2}$. First of all we need to deal with strata, which does not lie in $\partial\Delta_{4,2}$.

\begin{prop}
	For the stratum $W_{ij,ik,il, im,jk}$ the corresponding space of parameters $\tilde{F}_{ij,ik,il, im,jk}$ is follows:  the cross-ratio $w_{j,k,l,m}$ is equal $[1:1]$ and other cross-ratios are arbitrary.
	For the stratum $W_{ij,ik,il, im,jk,lm}$ $\tilde{F}_{ij,ik,il, im,jk,lm}$ is equal $\tilde{F}_{ij,ik,il, im,jk}$
\end{prop}

\begin{proof}[Proof]
This proposition is a straghtforward corollary of definition of $\tilde{F}_{\mathcal{I}}$.
\end{proof}

 For $W_{ij,ik,il, im,jl}$ and $W_{ij,ik,il, im,jm}$ other strata the answer is the same, but $w_{j,k,l,m} = [0:1]$ and $w_{j,k,l,m} = [1:0]$. 

\begin{prop}
	For any other stratum $W_{\mathcal{I}}$ the universal space of parameters $\tilde{F}_{\mathcal{I}}$ is equal $\mathcal{F}_{5,2}$.
\end{prop}

\begin{proof}[Proof]
	There is no admissible cross-ratio for the other stratum. By definition \ref{defn11}, it holds $\tilde{F}_{\mathcal{I}} = \mathcal{F}_{5,2}$.
\end{proof}

Now we can proof the following theorem.

\begin{thm}
	In the case of $G_{5,2}$ for any stratum $W_{\mathcal{I}}$ the correspnding universal space of parameters $\mathcal{F}_{\mathcal{I}}$ in the sense of definition \ref{defn11} is the same as universal spaces of parameters, which is constructed by Buchstaber and Terzic in \cite{BT1}.
\end{thm}

\begin{proof}[Proof]
	This is just comparison of all lists of universal spaces of parameters above and from \cite{BT1}. But all spaces
\end{proof}

~\\


\begin{thebibliography}
	{1}\bibitem{BT1} Victor M. Buchstaber, Svjetlana Terzic, \emph{Toric topology of the complex Grassmann manifolds}.  arXiv:1802.06449v2. 
	\bibitem{BT2} Victor M. Buchstaber, Svjetlana Terzic, \emph{The foundations of (2n,k)-manifolds}.  arXiv:1803.05766v1.
	\bibitem{BT3} Victor M. Buchstaber, Svjetlana Terzic, \emph{Topology and geometry of the canonical action of T4 on the complex Grassmannian $G_{4,2}$ and the complex projective space $\cp^5$}.  arXiv:1410.2482v3. 
	\bibitem{BP} Victor M. Buchstaber, Taras E.Panov, \emph{Toric topology}.  AMS, Rhode Island, 2015.	
	\bibitem{Kap}  M.Kapranov \emph{Chow quotients of Grassmannian I.}.arXiv:alg-geom/9210002v1.
	\bibitem{MS}  Dusa McDuff, Dietmar Salamon \emph{J-Holomorphic Curves and Symplectic Topology.}AMS, Rhode Island, 2004.
	\bibitem{GM}  I.M.Gelfand, R.W.MacPherson \emph{Geometry in Grassmannians and a generalization of the dilogarithm.}Adv. in Math. 44(1982), 279-312..
	\bibitem{Ke}  S.Keel \emph{Intersection theory of moduli space of stable N-pointed curves of genus zero.}Trans. of AMS, Vol. 330, No.2 (Apr. 1992), pp. 545-574.
	\bibitem{GH} P.Griffiths, J.Harris \emph{Principles of Algebraic Geometry} 1978 John Wiley \& Sons, Inc.
	\bibitem{Fltn} W.Fulton \emph{Young Tableaux
		With Applications to Representation Theory and Geometry} Cambridge: Cambridge University Press, 1977.
	\bibitem{Barlet}D.Barlet, \emph{Espace analytique reduit des cycles analytiques compexes compacts}, in:
	Lecture Notes in Mathematics, 482, p.1-158, Springer-Verlag, 1975
	
	
	
	
\end{thebibliography}
\end{document}